\documentclass[a4paper]{amsart}

\usepackage[all]{xy}
\usepackage{amsmath}
\usepackage{amssymb}
\usepackage{amsthm}
\usepackage{latexsym}
\usepackage{enumerate}
\usepackage{prettyref}
\usepackage{hyperref}

\newrefformat{conj}{\hyperref[{#1}]{Conjecture~\ref*{#1}}}
\newtheorem{theorem}{Theorem}[section]
\newrefformat{thm}{\hyperref[{#1}]{Theorem~\ref*{#1}}}
\newtheorem{definition}[theorem]{Definition}
\newrefformat{def}{\hyperref[{#1}]{Definition~\ref*{#1}}}
\newtheorem{lemma}[theorem]{Lemma}
\newrefformat{lem}{\hyperref[{#1}]{Lemma~\ref*{#1}}}
\newtheorem{proposition}[theorem]{Proposition}
\newrefformat{prop}{\hyperref[{#1}]{Proposition~\ref*{#1}}}
\newtheorem{corollary}[theorem]{Corollary}
\newrefformat{cor}{\hyperref[{#1}]{Corollary~\ref*{#1}}}
\newtheorem{remark}[theorem]{Remark}
\newrefformat{rem}{\hyperref[{#1}]{Remark~\ref*{#1}}}
{
\newtheorem{examplecore}[theorem]{Example}}
\newrefformat{ex}{\hyperref[{#1}]{Example~\ref*{#1}}}

\newenvironment{example}{\begin{examplecore}}{\hspace*{\fill}
$\square$\par\vspace{.1cm}\end{examplecore}}

\newcommand{\etale}{\ensuremath{\textrm{\'et}}}

\begin{document}

\title{On homology of linear groups over $k[t]$}
\author{Matthias Wendt}
\address{Matthias Wendt\\
Albert-Ludwigs-Uni\-ver\-si\-t\"at Freiburg, Mathematisches Institut\\
Eckerstra\ss{}e 1, 79104, Freiburg im Breisgau, Germany}
\email{matthias.wendt@math.uni-freiburg.de}


\begin{abstract}
This note explains how to prove that for any simply-connected
reductive group $G$ and any infinite field $k$, the inclusion
$k\hookrightarrow k[t]$ induces an isomorphism on group homology. This
generalizes results of Soul{\'e} and Knudson. 
\end{abstract}

\maketitle
\setcounter{tocdepth}{1}
\tableofcontents

\section{Introduction}

The question of homotopy invariance of group homology is the question
under which conditions on a linear algebraic group $G$ and a
commutative ring $R$ the natural morphism $G(R)\rightarrow G(R[t])$
induces isomorphisms in group homology. This is an unstable version of 
homotopy invariance for algebraic K-theory as established by Quillen
in \cite{quillen:1973:ktheory}. 

The two main results which have been obtained in this direction are
due to Soul{\'e} and Knudson. In \cite{soule:1979:polynomial},
Soul{\'e} determined a fundamental domain for the action of $G(k[t])$
on the associated Bruhat-Tits building and deduced homotopy invariance
for fields of characteristic $p>0$ with field coefficients prime to
$p$. In \cite{knudson:1997:polynomial}, Knudson extended Soul{\'e}'s
approach and deduced homotopy invariance with integral coefficients
for $SL_n$ over infinite fields.

In this paper, we generalize Knudson's theorem to arbitrary
(connected) reductive groups, cf. \prettyref{thm:knudson}: 

\begin{theorem}
\label{thm:main}
Let $k$ be an infinite field and let $G$ be a connected reductive
smooth linear algebraic group over $k$.  Then the canonical inclusion
$k\hookrightarrow k[t]$  induces isomorphisms
$$
H_\bullet(G(k),\mathbb{Z})\stackrel{\cong}{\longrightarrow}
H_\bullet(G(k[t]),\mathbb{Z}),
$$
if the order of the fundamental group of $G$ is invertible in $k$. 
\end{theorem}

It follows from the work of Krsti{\'c} and McCool \cite{krstic:mccool}
that homotopy invariance does not work for $H_1$ of rank one groups
over integral domains which are not fields. 
In the case of rank two groups, homotopy invariance fails for $H_2$ as
discussed in \cite{rk2}. It therefore seems that one cannot hope for
an extension of the above result for arbitrary regular rings or even
polynomial rings in more than $1$ variable.

\emph{Structure of the paper:} In \prettyref{sec:prelims}, we reduce
to simply-connected absolutely almost simple
groups. \prettyref{sec:building} recalls the 
necessary facts on Bruhat-Tits theory and Margaux's generalization of
Soul{\'e}'s theorem. In \prettyref{sec:knudson}, we extend the homology
computations of Knudson to groups other than $SL_n$. 

\emph{Acknowledgements:}  I have to apologize for the
numerous gaps in previous versions of this work. 
I would like to thank Mike Hopkins for pointing
out a critical error in an earlier version, and Christian Haesemeyer
and Aravind Asok for pointing out problems with Knudson's injectivity
argument in the rank two case. I also thank Aravind Asok, Christian
Haesemeyer, Roy Joshua and Fabien Morel for many useful comments and
their interest in the homotopy invariance question. 

\section{Preliminary reduction}
\label{sec:prelims}

In this section, we provide some preliminary reductions. More
precisely, \prettyref{thm:main} follows for all reductive groups  if
it can be shown for simply-connected almost simple  groups. The
arguments are fairly standard reductions, ubiquitous in the theory of
algebraic groups.

Recall from \cite{borel:1991:lag} the basic notions of the theory of
linear algebraic groups. In particular, the \emph{radical} $R(G)$ of a
group $G$ is the largest connected solvable normal subgroup of $G$,
and the \emph{unipotent radical} $R_u(G)$ of $G$ is the largest
connected unipotent normal subgroup of $G$. A connected group $G$ is called
\emph{reductive} if its unipotent radical is trivial, and it is called
\emph{semisimple} if its radical is trivial. 
A connected algebraic group $G$ is called \emph{simple} if it is 
non-commutative and has no nontrivial normal algebraic subgroups. It
is called \emph{almost simple} if its centre $Z$ is finite and the
quotient $G/Z$ is simple. 
A semisimple group is
called \emph{simply-connected}, if there is no nontrivial isogeny
$\phi:\tilde{G}\rightarrow G$.

The additive group is denoted by $\mathbb{G}_a$, and the
multiplicative group by $\mathbb{G}_m$. 
A \emph{torus} is a linear algebraic group $T$ defined over a field
$k$ which over the algebraic closure $\overline{k}$ is isomorphic to 
$\mathbb{G}_m^n$ for some $n$. 

We now show that the main theorem follows for reductive groups
if it can be proved for almost simple simply-connected groups. 
We first reduce to semisimple groups, the basic idea to keep in mind
is the sequence $SL_n\rightarrow GL_n\rightarrow \mathbb{G}_m$ which
reduces homotopy invariance for $GL_n$ to $SL_n$.

\begin{proposition}
To prove \prettyref{thm:main}, it suffices to consider the case where
$G$ is semisimple over $k$.
\end{proposition}

\begin{proof}
For a reductive group $G$, we have a split extension of linear 
algebraic groups 
\begin{displaymath}
1\rightarrow (G,G)\rightarrow G\rightarrow G/(G,G)\rightarrow 1
\end{displaymath}
where $(G,G)$ denotes the commutator subgroup in the sense of linear
algebraic groups. The quotient $G/(G,G)$ is  a  torus. 
We denote $H=(G,G)$ and $T=G/(G,G)$, and obtain a split exact sequence 
$1\rightarrow H(A)\rightarrow G(A)\rightarrow T(A)\rightarrow 1$ for
any $k$-algebra $A$. 
Assuming that $A$ is smooth and essentially of finite type, we have an
isomorphism $T(A)\cong T(A[t])$.
From the Hochschild-Serre spectral sequence for the above group
extensions we conclude that $G(k)\rightarrow G(k[t])$ induces an
isomorphism on homology if $H(k)\rightarrow H(k[t])$ induces an
isomorphism on homology. But $H=(G,G)$ is a semisimple algebraic group
over $k$. 
\end{proof}

\begin{proposition}
\label{prop:reduction}
To prove \prettyref{thm:main}, it suffices to consider the case where
$G$ is almost simple simply-connected over $k$.
\end{proposition}

\begin{proof}
For a semisimple group $G$, there is an exact
sequence of algebraic groups
\begin{displaymath}
1\rightarrow\Pi\rightarrow \widetilde{G}\rightarrow G\rightarrow 1,
\end{displaymath}
where $\Pi$ is a finite central group scheme, and
$\widetilde{G}$ is a product of simply-connected 
almost simple groups. If we assume that \prettyref{thm:main} holds for
these simply-connected almost simple groups, it is also true for their
product, by a simple application of the Hochschild-Serre spectral
sequence.  

Now assume that the order of $\Pi$ is prime to the characteristic
of the field $k$, as in \prettyref{thm:main}. 
From the universal covering above we have an exact sequence
\begin{displaymath}
1\rightarrow \Pi(R) \rightarrow \tilde{G}(R) \rightarrow G(R)
\rightarrow H^1_{\etale}(R,\Pi)\rightarrow H^1_{\etale}(R,\tilde{G})
\end{displaymath}
for any $k$-algebra $R$. Then we have isomorphisms
\begin{displaymath}
\Pi(k)\cong \Pi(k[t]), \qquad\textrm{ and }\qquad
H^1_{\etale}(k,\Pi)\cong H^1_{\etale}(k[t],\Pi),
\end{displaymath}
by our assumption on the characteristic of the base field. 

Now the first part of the exact sequence above yields a morphism of
exact sequences
\begin{center}
  \begin{minipage}[c]{10cm}
    \xymatrix{
      1 \ar[r] & \Pi(k) \ar[r] \ar[d]_\cong &
      \tilde{G}(k)\ar[r]\ar[d] &       \tilde{G}(k)/\Pi(k)
      \ar[r] \ar[d] &  1 \\ 
      1 \ar[r] & \Pi(k[t]) \ar[r] & \tilde{G}(k[t]) \ar[r] &
      \tilde{G}(k[t])/\Pi(k[t]) \ar[r] &  1
    }
  \end{minipage}
\end{center}
Since $\Pi$ is in fact abelian, one can consider fibre sequences
\begin{displaymath}
B\tilde{G}(R)\rightarrow B\left(\tilde{G}(R)/\Pi(R)\right)
\rightarrow K(\Pi(R),2)
\end{displaymath}
for $R=k$ and $R=k[t]$. Then the associated Hochschild-Serre spectral
sequence implies that the morphism $\tilde{G}(k)/\Pi(k)\rightarrow
\tilde{G}(k[t])/\Pi(k[t])$ induces an isomorphism on homology, since
we argued before that the morphism $\tilde{G}(k)\rightarrow
\tilde{G}(k[t])$ induces an isomorphism on homology.

Since  
\begin{displaymath}
H^1_{\etale}(k,\tilde{G})\rightarrow
H^1_{\etale}(k[t],\tilde{G})
\end{displaymath}
 is injective ($k$ is a retract of
$k[t]$), the images of $G(k[t])$ and $G(k)$ in
$H^1_{\etale}(k,\Pi)\cong H^1_{\etale}(k[t],\Pi)$  are
equal - we denote these images by $\pi_0(G(k[t]))$ and $\pi_0(G(k))$,
respectively. Therefore we get a morphism of extensions 
\begin{center}
  \begin{minipage}[c]{10cm}
    \xymatrix{
      1 \ar[r] & \tilde{G}(k[t])/\Pi(k[t]) \ar[r] \ar[d] & G(k[t])
      \ar[r] \ar[d] & \pi_0(G(k[t])) \ar[r] \ar[d]^\cong & 1 \\
      1 \ar[r] & \tilde{G}(k)/\Pi(k) \ar[r] & G(k)
      \ar[r] & \pi_0(G(k)) \ar[r] & 1
    }
  \end{minipage}
\end{center}
The outer vertical arrows are isomorphisms on homology, therefore the
comparison theorem for the Hochschild-Serre spectral sequences implies
that we obtain an isomorphism on the middle arrow.
\end{proof}

Henceforth, we shall only consider linear algebraic groups $G$,
defined over $k$ which are almost simple and simply-connected. 

\section{Bruhat-Tits buildings and the Soul{\'e}-Margaux-theorem}
\label{sec:building}

In this section, we recall the basics of the theory of buildings which
will be needed in the remaining sections. The main references are
\cite{bruhat:tits:1972:building} and \cite{abramenko:brown}. 

Let $k$ be
a field. Then we equip the function field $K=k(t)$ with the valuation
$\omega_\infty(f/g)=\operatorname{deg}(g)-\operatorname{deg}(f)$,
with $t^{-1}$ as uniformizer. We denote by $\mathcal{O}$ the corresponding
discrete valuation ring. Alternatively, one can work with
$K=k((t^{-1}))$ and the corresponding valuation ring
$k[[t^{-1}]]$. The underlying simplicial complex of the building will
be the same, only the apartment system will be different. 

Let $G$ be a reductive group over $k$. Then we have two morphisms of
groups, the inclusion $G(\mathcal{O})\hookrightarrow G(K)$ and the
reduction $G(\mathcal{O})\rightarrow G(k)$.  

\subsection{BN-Pairs and buildings}

We will be concerned with affine buildings associated to reductive
groups over discretely valued fields. We recall the definition of
buildings based on the notion of BN-pairs. This theory is detailed in
\cite{abramenko:brown}, in particular Section 6.

\begin{definition}
A pair of subgroups $B$ and $N$ of a group $G$ is called a
\emph{BN-pair} if $B$ and $N$ generate $G$, the intersection $T:=B\cap
N$ is normal in $N$, and the quotient $W=N/T$ admits a set of
generators $S$ such that the following two conditions hold:
\begin{enumerate}[(BN1)]
\item For $s\in S$ and $w\in W$ we have $sBw\subseteq BswB\cup BwB$. 
\item For $s\in S$, we have $sBs^{-1}\nleq B$.
\end{enumerate}
The group $W$ is called the \emph{Weyl group} of the BN-pair. The
tuple $(G,B,N,S)$ is called \emph{Tits system}. 
\end{definition}

We now describe the BN-pair on $G(K)$ which will be relevant for
us. We mostly stick to the notation used in \cite{soule:1979:polynomial}. 
Choose a maximal torus $T\subseteq G$. This fixes two subgroups 
$T(k)\subseteq G(k)$ and $T(K)\subseteq G(K)$. 
Fix a choice of Borel subgroup $\overline{B}$ in $G(k)$ containing
$T(k)$. 
\begin{quote}
For the definition of the BN-pair, we let $B\subseteq G(K)$ be
the preimage of $\overline{B}$ under the reduction
$G(\mathcal{O})\rightarrow G(k)$. The group $N$ is defined as the
normalizer of $T(K)$ in $G(K)$.
\end{quote}
This is the usual construction, explained in detail for the case
$SL_n$ in \cite[Section 6.9]{abramenko:brown}. We will not recall the
proof that this indeed yields a BN-pair here.

We recall one particular description of the building associated to a
BN-pair from \cite[Section 6.2.6]{abramenko:brown}. Given a Tits
system $(G,B,N,S)$, a subgroup $P\subseteq G$ is called
\emph{parabolic} if it contains a conjugate of $B$. The subgroups of
$G$ which contain $B$ are called \emph{standard parabolic
  subgroups}. These are associated to subsets of $S$.  

The building $\Delta(G,B)$ for $(G,B,N,S)$ is the simplicial complex
associated to the ordered set of parabolic subgroups of $G$, ordered
by reverse inclusion. The group $G$ acts via conjugation. The
fundamental apartment is given by 
\begin{displaymath}
\Sigma=\{wPw^{-1}\mid w\in W, P\geq B\}.
\end{displaymath}
The other apartments are of course obtained by using conjugates of the
group $B$ above. Alternatively, the building can be described as the
simplicial complex associated to the ordered set of cosets of the
standard parabolic subgroups, with the group $G$ acting via
multiplication. 

\subsection{Soul{\'e}'s fundamental domain}

We continue to consider the BN-pair defined above. 
In the standard apartment $\Sigma$ of $\Delta(G,B)$, there is one
vertex fixed by $G(\mathcal{O})$. This vertex is denoted by
$\phi$. The fundamental chamber containing the vertex $\phi$ is given
by 
\begin{displaymath}
\mathcal{C}=\{P\mid P\geq B\}\subseteq \Sigma.
\end{displaymath}
The fundamental sector $\mathcal{Q}$ is the simplicial cone with
vertex $\phi$ which is generated by $\mathcal{C}$. 

The following theorem was proved in \cite{soule:1979:polynomial} and
subsequently generalized to isotropic simply-connected absolutely
almost simple groups, cf. \cite{margaux:2008:soule}.  

\begin{theorem}
\label{thm:soule}
The set $\mathcal{Q}$ is a simplicial fundamental domain for the
action of $G(k[t])$ on the Bruhat-Tits building $\Delta(G,B)$. In 
other words, any simplex of $\Delta(G,B)$ is equivalent under the
action of $G(k[t])$ to a unique simplex of $\mathcal{Q}$.
\end{theorem}

\subsection{Stabilizers}

We are also interested in the subgroups which stabilize simplices in
the fundamental domain $\mathcal{Q}$. We have seen above that the
simplices correspond to standard parabolic subgroups. It turns out
that the stabilizer of the simplex corresponding to $G\geq P\geq B$ is
exactly $P$, cf. \cite[Theorem 6.43]{abramenko:brown}. In particular,
for the group $G(k[t])$, we find that the stabilizer of a simplex
$\sigma_P$ corresponding to a parabolic subgroup $P$ of $G(K)$ is
the following intersection
\begin{displaymath}
\operatorname{Stab}(\sigma(P))=G(k[t])\cap P.
\end{displaymath}

This implies a concrete description of the stabilizers,
cf. \cite[Paragraph 1.1]{soule:1979:polynomial}
resp. \cite[Proposition 2.5]{margaux:2008:soule}.  

\begin{proposition}
\label{prop:stabilizer}
Let $x\in \mathcal{Q}\setminus\{\phi\}$. We denote by
$\operatorname{Stab}(x)$ the stabilizer of $x$ in $G(k[t])$. 
\begin{enumerate}[(i)]
\item There is an extension of groups
\begin{displaymath}
1\rightarrow \operatorname{Stab}(x)\cap U_x(K)\rightarrow
\operatorname{Stab}(x) \rightarrow L_x(k)\rightarrow 1. 
\end{displaymath}
The group $L_x(k)$ is a reductive subgroup of $G(k)$, in fact it
is a Levi subgroup of a maximal parabolic subgroup of $G(k)$ for the
spherical BN-pair. The group $\operatorname{Stab}(x)\cap U_x(K)$ is a
split unipotent subgroup of $U_x(k[t])$.  
\item
The stabilizer of a simplex $\sigma$ is the intersection of the
stabilizers of the vertices $x$ of $\sigma$.
\item The stabilizers can be described using the valuation of the root
  system, cf. \cite[Section 1.1]{soule:1979:polynomial}. 
  In particular, in the notation of Soul{\'e}, we have
  \begin{displaymath}
    \Gamma_x=L_x(k)\cdot U_x(k[t]), \qquad L_x(k)=T(k)\cdot\langle
    x_\alpha(k)\mid \alpha(x)=0\rangle,
  \end{displaymath}
  \begin{displaymath}
    U_x(k[t])=\langle x_\alpha(u), u\in k[t], d\circ(u)\leq \alpha(x),
    \alpha(x)>0 \rangle.
  \end{displaymath}
  Withouth explaining all the notation in detail, this means that an
  element of $Z_x(k)=L_x(k)$ is a product of an element of the torus
  and certain root elements, where the roots only depend on the vertex
  $x$. An element in $U_x(k[t])$ is a product of certain root elements,
  the degree of the polynomials and the roots only depend on the
  vertex $x$. 
\end{enumerate}
\end{proposition}

\section{Homology of the stabilizers and Knudson's theorem}
\label{sec:knudson}

In this section, we describe the homology of the stabilizers of
simplices in Soul{\'e}'s domain. In \cite{knudson:1997:polynomial},
Knudson showed in the case $SL_n$ that the homology of the stabilizers
is determined by the homology of a Levi subgroup. We provide below a
generalization of this result. The results work in general for rings
with many units.

From \prettyref{prop:stabilizer}, we know that for a simplex $\sigma$
in the fundamental domain $\mathcal{Q}$, the stabilizer
$\Gamma_\sigma=\operatorname{Stab}(\sigma)$ of $\sigma$ in
$G(k[t])$ sits in an extension
\begin{displaymath}
1\rightarrow U_\sigma\rightarrow \Gamma_\sigma\rightarrow
L_\sigma\rightarrow 1, 
\end{displaymath}
where $U_\sigma$ is an abstract group contained in a unipotent
subgroup of $G(k[t])$ and $L_\sigma$ is the group of $k$-points of a
reductive  subgroup of $G$. The first thing we will show in this
section is that the induced morphism
$H_\bullet(\Gamma_\sigma,\mathbb{Z}) \rightarrow
H_\bullet(L_\sigma,\mathbb{Z})$ is an isomorphism. 
This is done via the Hochschild-Serre spectral 
sequence 
\begin{displaymath}
E^2_{p,q}=H_p(L_\sigma,H_q(U_\sigma))\Rightarrow H_{p+q}(\Gamma_\sigma)
\end{displaymath}
associated to the above group extension. To show the result, it
suffices to show that $H_p(L_\sigma,H_q(U_\sigma))=0$ for $q>0$. 

The basic idea for showing this latter assertion is to use the action
of $k^\times$ on the group $U_\sigma$, where $k^\times$ is embedded in
$L_\sigma$ as the $k$-points of a suitable subtorus. The group
$k^\times$ acts via multiplication on the various abelian subquotients
constituting the unipotent group $U_\sigma$, and an argument as in
\cite[Theorem 2.2.2]{knudson:2001:linear} shows that this homology is
trivial. This argument is detailed after some introductory remarks in
\prettyref{thm:stabilizer}. 

\subsection{A result of Suslin} A ring $A$ is an $S(n)$-ring if there
are $a_1,\dots,a_n\in A^\times$ such that the sum of each nonempty
subfamily is a unit. If $A$ is an $S(n)$-ring for all $n$, then $A$ is
said to \emph{have many units}. 

As explained in \cite[Section 2.2.1]{knudson:2001:linear}, the right
way to prove that 
\begin{displaymath}
H_p(GL_n(A),H_q(M_{n,m}(A),\mathbb{Z}))=0
\end{displaymath}
for $q>0$ if $A$ is a $\mathbb{Q}$-algebra is the following: we notice
that $M_{n,m}(A)$ is an abelian group, and therefore 
$H_q(M_{n,m}(A))=\bigwedge^qM_{n,m}(A)$. There exists a central
element $a\in GL_n(A)$ which acts on $M_{n,m}(A)$ by multiplication
with $a$, and therefore by multiplication with $a^q$ on
$H_q(M_{n,m}(A))$. This action is trivial, because $a$ is in the
center, and therefore $H_q(M_{n,m}(A))$ is annihilated by $a^q-1$. But
it is a $\mathbb{Q}$-vector space and therefore it is trivial for $q>0$. 

The following result due to
Nesterenko and Suslin \cite{nesterenko:suslin:1990:stability} is a
generalization of this center-kills-argument to rings with many units
in arbitrary characteristics. For more information on the proof of
this result, cf. \cite[Section 2.2.1]{knudson:2001:linear}.

\begin{proposition}
\label{prop:vanish}
Let $A$ be a ring with many units, and let $F$ be a prime field. Then
for all $i\geq 0$ and $j>0$, we have $H_i(A^\times,H_j(A^s,F))=0$,
where $A^\times$ acts diagonally on $A^s$. 
\end{proposition}

The same conclusion also obtains for actions of $A^\times$ via
non-zero powers of units, cf.  \cite[Lemma
9]{hutchinson:1990:matsumoto}.  

\subsection{Example: orthogonal groups} We explain the procedure using 
the special case of orthogonal groups from
\cite{vogtmann:1979:stability}. For the groups $O_{n,n}$ over a field
$k$ of characteristic $\neq 2$, there are maximal parabolic subgroups
$P_I$ which have a non-abelian unipotent radical. They have the
following general form, 
cf. \cite[p. 21]{vogtmann:1979:stability}:  
\begin{displaymath}
\left(\begin{array}{ccc}
A & \ast & \ast \\
0 & B & \ast \\
0 & 0 & {}^tA^{-1}
\end{array}\right),
\end{displaymath}
where $A\in GL_p(k)$, $B\in O_{n-p,n-p}(k)$ and there are some
additional conditions on the $\ast$-terms ensuring that the whole
matrix is in $O_{n,n}(k)$. It is proved on p.34 of that paper that the
unipotent subgroup $N$ of $P_I$ sits in an exact sequence
\begin{displaymath}
1\rightarrow [N,N]\rightarrow N\rightarrow N/[N,N]\rightarrow 1
\end{displaymath}
with the outer terms $[N,N]$ and $N/[N,N]$ abelian groups. It is also
proved that the torus
\begin{displaymath}
\operatorname{diag}(\underbrace{a,\dots,a}_{p},\underbrace{1,\dots,1}_{2n-2p},
\underbrace{a^{-1},\dots,a^{-1}}_{p}) 
\end{displaymath}
acts via multiplication with $a$ on $N/[N,N]$ and multiplication with
$a^2$ on $[N,N]$.

We apply the Hochschild-Serre spectral sequence for the extension
\begin{displaymath}
1\rightarrow [N,N]\rightarrow P_I\rightarrow P_I/[N,N]\rightarrow 1.
\end{displaymath}
This has the following form: 
\begin{displaymath}
H_p(P_I/[N,N],H_q([N,N]))\Rightarrow H_{p+q}(P_I).
\end{displaymath}
To prove that $H_p(P_I/[N,N],H_q([N,N]))=0$ for $q>0$, we use another
Hochschild-Serre spectral sequence for the torus action:
\begin{displaymath}
H_p(P_I/[N,N]/k^\times,H_j(k^\times,H_q([N,N]))
\Rightarrow H_{p+j}(P_I/[N,N],H_q([N,N])).
\end{displaymath}
Since the action of $k^\times$ on $[N,N]$ is via multiplication by
squares, the result of Suslin, cf. \prettyref{prop:vanish}, implies
$H_p(P_I/[N,N],H_q([N,N]))=0$ for $q>0$.

The morphism $P_I\rightarrow P_I/[N,N]$ thus induces an isomorphism on
homology. A similar argument applied to the extension
\begin{displaymath}
1\rightarrow N/[N,N]\rightarrow P_I/[N,N]\rightarrow P_I/N\rightarrow 1
\end{displaymath}
implies that the morphism $P_I/[N,N]\rightarrow P_I/N$ also induces an
isomorphism on homology. 

This amounts to a proof of \cite[Proposition
2.2]{vogtmann:1979:stability} for  infinite fields of characteristic
$p\neq 2$. We obtain the following strengthening of Vogtmann's
stability result, making explicit a remark in \cite[Section
2.4.1]{knudson:2001:linear}.  

\begin{corollary}
\label{cor:vogtmann}
Let $k$ be an infinite field of characteristic $\neq 2$. Then the
induced morphism
\begin{displaymath}
H_i(O_{n,n}(k),\mathbb{Z})\rightarrow
H_i(O_{n+1,n+1}(k),\mathbb{Z}) 
\end{displaymath}
is surjective for $n\geq 3i+1$ and an isomorphism for $n\geq 3i+3$. 
\end{corollary}

\begin{remark}
\begin{enumerate}[(i)]
\item The above stabilization result is the one obtained via
  Vogtmann's argument in \cite{vogtmann:1979:stability} using the
  improved computation of the homology of stabilizers. Better
  stabilization results for orthogonal groups are available in the
  work of Essert, cf. \cite{essert}. In Essert's work, dealing with
  the homological contribution from the unipotent radical as above is
  not necessary - he uses opposition complexes, where stabilizers are
  Levi subgroups instead of the full parabolic subgroups.
\item
The above example for orthogonal groups is an instance of a more
general result which can be found e.g. in \cite{azad:barry:seitz}. Let
$G$ be a reductive group over $k$, let $\Phi$ be the associated
root system and assume that $\operatorname{char}k$ is not equal to $2$
for $\Phi$ doubly laced resp. not equal to $2$ or $3$ for $\Phi$
triply laced. Let $P$ be a parabolic subgroup associated to a
subset $I\subseteq\Phi$ of simple roots, and let $U$ be the unipotent
radical of $P$. Then the length of the descending central series of
$U$ equals $\sum_{\alpha\in I}m(\alpha)$ where $\alpha$ is the
multiplicity of $\alpha$ in the highest root $\widetilde{\alpha}$ of
$\Phi$. 

In the above example of orthogonal groups, the root system is of type
$D_n$. Numbering the simple roots
$\alpha_1,\dots,\alpha_{n-1},\alpha_n$ such that $\alpha_1$,
$\alpha_{n-1}$ and $\alpha_n$ correspond to the end-points of the
Dynkin diagram, the longest root is
$$
\widetilde{\alpha}=
\alpha_1+2\alpha_2+\cdots+2\alpha_{n-2}+\alpha_{n-1}+\alpha_n.
$$
The parabolic subgroups discussed in the above example are the ones
corresponding to roots $\alpha_2,\dots,\alpha_{n-2}$. 
\end{enumerate}
\end{remark}

\subsection{Homology of the stabilizers} We now want to compute the
homology of the stabilizers. We formulate the proof with the
$k$-points of the stabilizers, where $k$ is an infinite field. The
same arguments show the result more generally for an integral
domain $R$ with many units having quotient field $k$. The goal is to
compute the homology of the stabilizer $\Gamma_\sigma(k[t])$. The Levi
subgroup $L_\sigma$ is defined as 
$L_\sigma=Z_G(S_\sigma)$, i.e. as the centralizer of a split torus
$S_\sigma$ in $G$ associated to the simplex $\sigma$. We note that it
follows from this definition that there is a normal central torus in
$L_\sigma$. 
Now consider a subtorus $\mathbb{G}_m\rightarrow L_\sigma$. If the
corresponding abstract group $k^\times$ acts trivially on the
$k$-points of a unipotent subgroup $U\subseteq G(k[t])$, then
$U\subseteq Z_G(\mathbb{G}_m)$ and hence $U$ is contained in
$L_\sigma$. Therefore, for any unipotent subgroup  $U\subseteq
\Gamma_\sigma$ which is not contained in $L_\sigma$, there has to
exist a torus $\mathbb{G}_m\subseteq L_\sigma$ such that the
corresponding group of $k$-points $k^\times$ acts nontrivially on the
$k$-points of $U$.    

Note that the unipotent radicals $U_\sigma$ of parabolic subgroups of
$G$ associated to the simplex $\sigma$ are actually split, i.e. there
is a filtration  
\begin{displaymath}
U_\sigma=U_1\supseteq U_2\supseteq\cdots\supseteq U_n=\{1\}
\end{displaymath}
with each $U_n/U_{n+1}$ being isomorphic to $\mathbb{G}_a$. Since the
automorphism group of $\mathbb{G}_a$ is $\mathbb{G}_m$, the
multiplicative group $\mathbb{G}_m$ can only act via 
\begin{displaymath}
(a\in k^\times,u\in k)\mapsto a^n u
\end{displaymath}
for some $n$.
We have thus established the following: 

\begin{lemma}
\label{lem:action}
Let $U_\sigma$ be the unipotent radical of the stabilizer
$\Gamma_\sigma$, and let 
\begin{displaymath}
U_\sigma=U_1\supseteq U_2\supseteq\cdots\supseteq U_n=\{1\}
\end{displaymath}
be a filtration such that $U_i/U_{i+1}\cong\mathbb{G}_a$. For each
$i$, there exists a central embedding $k^\times\rightarrow L_\sigma$
and a number $n_i$ such that $a\in k^\times$ acts on $U_i/U_{i+1}$ via
multiplication with $a^{n_i}$.
\end{lemma}

Note that above we are only talking about algebraic groups, i.e. about
the unipotent radical of parabolic subgroups of $G(k(t))$. However,
since the $k[t]$-points of the torus are $k^\times$, the action
preserves the degree filtration of the $k[t]$-points of unipotent
radicals. In particular, the action described above restricts to an
action of $k^\times$ on the unipotent part $U_\sigma$ of the
stabilizer subgroup $\Gamma_\sigma$, for any simplex
$\sigma\in\mathcal{Q}$. 

\begin{example}
\begin{enumerate}[(i)]
\item
The simplest example of this situation is the embedding 
\begin{displaymath}
R^\times\rightarrow SL_{n+m}:a\mapsto
\operatorname{diag}(\underbrace{a^m,\dots,a^m,}_n
\underbrace{a^{-n},\dots, a^{-n}}_m).
\end{displaymath}
The centralizer of this torus is the Levi subgroup of a maximal
parabolic subgroup which is the intersection of the following subgroup
with $SL_{n+m}$:
\begin{displaymath}
\left(
\begin{array}{cc}
GL_n & 0 \\ 0 & GL_m
\end{array}
\right).
\end{displaymath}
The corresponding parabolic subgroup has the form 
\begin{displaymath}
\left(
\begin{array}{cc}
GL_n & M \\ 0 & GL_m
\end{array}
\right)\cap SL_{n+m}
\end{displaymath}
and the torus acts on $M$ via multiplication with $a^{m+n}$,
cf. \cite{hutchinson:1990:matsumoto}. 
\item
Another example of such a situation is the one discussed in the proof
of \prettyref{cor:vogtmann}, cf. also
\cite[p. 34]{vogtmann:1979:stability}. In these cases, the unipotent
radical of a maximal parabolic of a split orthogonal or symplectic
group is not abelian, and the torus acts via different
powers on the steps of the central series.
\end{enumerate}
\end{example}

The above actions now allow to compute the $E_2$-term of the
Hochschild-Serre spectral sequence.
This is done by using the composition series of $U_I$, which induces a
sequence 
\begin{displaymath}
\Gamma_I\rightarrow \Gamma_I/\mathbb{G}_a=\Gamma_I/U_n\rightarrow
\Gamma_I/U_2\rightarrow \cdots\rightarrow \Gamma_I/U_I=L_I.
\end{displaymath}
We will show below that each step induces isomorphisms on homology. The
argument is a generalization of the proof of \cite[Corollary 
3.2]{knudson:1997:polynomial}. 

The following theorem now describes the homology of the stabilizers of
the action of $G(k)$ on  the Bruhat-Tits building. 

\begin{theorem}
\label{thm:stabilizer}
Let $R$ be an integral domain with many units and denote by $k=Q(R)$
its field of fractions. 
The group $G(R[t])$ acts on the Bruhat-Tits building associated
to the group $G(k(t))$, 
and we consider the stabilizer group $\Gamma_\sigma$ of a simplex
$\sigma\in\mathcal{Q}$. Then the morphism
\begin{displaymath}
H_\bullet(\Gamma_\sigma,\mathbb{Z})\rightarrow
H_\bullet(L_\sigma,\mathbb{Z})
\end{displaymath}
induced from the projection in \prettyref{prop:stabilizer} is an
isomorphism. 
\end{theorem}

\begin{proof}
Consider the composition series of
$U_\sigma$: 
\begin{displaymath}
U_\sigma=U_1\supseteq U_2\supseteq\cdots\supseteq U_n=\{1\}.
\end{displaymath}
More precisely, the composition series of the unipotent group as an
algebraic group induces a similar filtration of the unipotent part of
the stabilizer, which is defined inside the unipotent radical by
degree bounds as in \prettyref{prop:stabilizer}.
This induces a sequence of group homomorphisms 
\begin{displaymath}
\Gamma_\sigma\rightarrow \Gamma_\sigma/U_n\rightarrow
\Gamma_\sigma/U_2\rightarrow \cdots\rightarrow \Gamma_\sigma/U_\sigma=L_\sigma.
\end{displaymath}
Each step in this sequence is a quotient by a subgroup
of $\mathbb{G}_a(R[t])$ in $\Gamma_\sigma/U_i$. 
It therefore suffices to show that each such morphism induces an
isomorphism on homology. This is done via the Hochschild-Serre
spectral sequence, which then looks like
\begin{displaymath}
H_p(\Gamma_\sigma/U_{i+1},H_q(U_i/U_{i+1},\mathbb{Z}))\Rightarrow 
H_p(\Gamma_\sigma/U_i,\mathbb{Z}). 
\end{displaymath}
Thus it suffices to show for any prime field $F$, we have
\begin{displaymath}
H_p(\Gamma_\sigma/U_{i+1},H_q(U_i/U_{i+1},F))=0
\end{displaymath}
for $q>0$. But by \prettyref{lem:action}, there is a central
embedding $R^\times\rightarrow \Gamma_\sigma/U_{i+1}$ such that $a\in
R^\times$ acts on $U_i/U_{i+1}$ via multiplication by some non-zero
power of $a$. 

We have an associated  Hochschild-Serre spectral sequence
$$
E^2_{j,l}=H_j((\Gamma_\sigma/U_{i+1})/R^\times,H_l(R^\times,H_q(U_i/U_{i+1},F)))
\Rightarrow
$$
$$\Rightarrow
H_{j+l}(\Gamma_\sigma/U_{i+1},H_q(U_i/U_{i+1},F)).
$$
From \prettyref{prop:vanish}, we obtain that
$H_l(R^\times,H_q(U_i/U_{i+1},F))=0$ for $q>0$, which finishes
the proof.
\end{proof}

\subsection{The theorem of Knudson}
We will now  prove homotopy invariance in the one-variable case. The
following is a generalization of \cite[Corollary
4.6.3]{knudson:2001:linear}.

\begin{theorem}
\label{thm:knudson}
Let $k$ be an infinite field and let $G$ be a connected reductive
group over $k$. Then the inclusion $k\hookrightarrow k[t]$ induces an
isomorphism  
$$
H_\bullet(G(k),\mathbb{Z})\stackrel{\cong}{\longrightarrow}
H_\bullet(G(k[t]),\mathbb{Z}),
$$
if the order of the fundamental group of $G$ is invertible in $k$. 
\end{theorem}

\begin{proof}
By \prettyref{prop:reduction}, we can assume that $G$ is
simply-connected absolutely almost simple over $k$. Then we can use
the action of the group $G(k[t])$ on the building associated to
$G(k(t))$. By \prettyref{thm:soule}, the subcomplex $\mathcal{Q}$ is a
fundamental domain for this action.
There is an associated spectral sequence
\begin{displaymath}
E^1_{p,q}=\bigoplus_{\dim\sigma=p,\sigma\in\mathcal{Q}}
H_q(\Gamma_\sigma,\mathbb{Z})\Rightarrow 
H_{p+q}(G(k[t]),\mathbb{Z}).
\end{displaymath}
In the above,
$\Gamma_\sigma$ is the stabilizer of the simplex $\sigma$.  
This is the spectral sequence computing the $G(k[t])$-equivariant
homology of the building, 
cf. \cite[p. 162]{knudson:2001:linear}. 

From \prettyref{thm:stabilizer}, we know the homology of the
stabilizers,  in particular, that it only depends on the reductive
part. In the notation of \cite{knudson:1997:polynomial}, there is a
filtration of the fundamental domain $\mathcal{Q}$ via subsets
$E_I^{(k)}$ for any $k$-element subset $I$ of roots of $G$. 
These subsets are simplicial subcones of $\mathcal{Q}$ which consist
of all simplices of $\mathcal{Q}$ such that the constant part of the
stabilizer is the standard parabolic subgroup of $G$
determined by the subset $I$. 
This yields a filtration 
\begin{displaymath}
\mathcal{Q}^{(k)}=\bigcup_IE^{(k)}_I. 
\end{displaymath}
Now for any two simplices $\sigma$, $\tau$ in 
\begin{displaymath}
E_I^{(k)}\setminus\bigcup_{J\subset I} E_J^{(k-1)},
\end{displaymath}
the stabilizers $\Gamma_\sigma$ and $\Gamma_\tau$ have the same
reductive part $L_\sigma=L_\tau$, and therefore they also have the
same homology, by \prettyref{thm:stabilizer}.  The coefficient system
$\sigma\mapsto H_q(\Gamma_\sigma)$ is then locally constant in the
sense of \cite[Proposition A.2.7]{knudson:2001:linear}, and we obtain
an isomorphism 
\begin{displaymath}
H_\bullet(\phi,\mathcal{H}_q)\rightarrow
H_\bullet(\mathcal{Q},\mathcal{H}_q). 
\end{displaymath}
This shows that the argument in the proof of \cite[Theorem 
3.4]{knudson:1997:polynomial} does not depend on $SL_n$. Therefore,
\cite[Proposition A.2.7]{knudson:2001:linear} implies that the
$E_2$-term of the above spectral sequence looks as follows: 
\begin{displaymath}
E^2_{p,q}=\left\{\begin{array}{cc}
H_q(G(k),\mathbb{Z}) & p = 0\\
0 & p>0
\end{array}
\right.
\end{displaymath}
The spectral sequence degenerates and the result is proved.
\end{proof}

\begin{remark}
\prettyref{thm:stabilizer} has
been used in \cite{rk2} to establish homotopy invariance for the
homology of Steinberg groups of rank two groups. 
Apart from this, it seems that the
  added generality of rings with many units in
  \prettyref{thm:stabilizer} can not be widely
  applied. Generalizations of \prettyref{thm:knudson} beyond the case
  $k[t]$ seem to be generally wrong. The failure of homotopy
  invariance for $H_1$ of $SL_2$ follows directly from
  \cite{krstic:mccool}. The failure of homotopy invariance for $H_2$
  of rank two groups has been established in \cite{rk2}. In these
  cases, one sees that $\mathcal{Q}$ fails quite badly to be a
  fundamental domain for the action of $G(R[t])$ if $R$ is not a field
  -- in case $SL_2$, the subcomplex $SL_2(R[t])\cdot\mathcal{Q}$ is
  not connected and in case $SL_3$, the subcomplex
  $SL_3(R[t])\cdot\mathcal{Q}$ is not simply-connected. 
\end{remark}

Concerning the subcomplex $G(R[t])\cdot\mathcal{Q}$ for $R$ an
integral domain, we have the following: 

\begin{proposition}
\begin{enumerate}[(i)]
\item The complex $E(R[t])\cdot\mathcal{Q}$ is connected. 
\item The complex
  $G(R[t])\cdot\mathcal{Q}$ is connected if $G(R[t])=E(R[t])\cdot
  G(R)$. This in particular holds for isotropic reductive groups $G$
  of rank $\geq 2$ and $R$ essentially smooth over a field. 
\item The complex $E(R[t])\cdot\mathcal{Q}$ is simply-connected if
  $K_2^G(R[t])\cong K_2^G(R)$, in particular for $G=SL_n$, $n\geq 5$
  and $R=k[t_1,\dots,t_m]$.
\end{enumerate}
\end{proposition}

\begin{proof}
Every elementary matrix for a positive root is contained in some
stabilizer, and the stabilizer of $\phi$ contains the Weyl group. By
\cite[Theorem 2]{soule:1979:polynomial} the complex
$E(\Phi,R[t])\cdot\mathcal{Q}$ is connected, hence (i). The
same argument shows (ii). The additional consequence in (ii) is the
work of Suslin \cite{suslin:1977:serre}, Abe \cite{abe:1983:whitehead}
and in the non-split case Stavrova \cite{stavrova}. 
We only sketch  (iii): it follows from the assumption on homotopy
invariance of $K_2$ that $E(R[t])$ is an amalgam of the stabilizers
along their intersections. Again \cite[Theorem
2]{soule:1979:polynomial} shows the claim. The additional assertion
for $SL_n$ is a consequence of \cite{tulenbaev}.
\end{proof}

\end{document}